\documentclass{amsart}
\usepackage{graphicx}
\usepackage{amsmath}
\usepackage{epstopdf}

\newtheorem{thm}{Theorem}[section]

\newtheorem{claim}[thm]{Claim}
\newtheorem{lemma}[thm]{Lemma}
\newtheorem{prop}[thm]{Proposition}
\theoremstyle{definition}
\newtheorem{defn}[thm]{Definition}
\theoremstyle{remark}

\newtheorem{remark}[thm]{Remark}
\numberwithin{equation}{section}

\newcommand{\Integ}{\mathbb Z}
\newcommand{\B}{\langle B \rangle}
\newcommand{\C}{\langle C \rangle}
\newcommand{\SG}{\mathcal{S}_\Gamma}
\newcommand{\SB}{\mathcal{S}_B}
\newcommand{\SC}{\mathcal{S}_C}

\newcommand{\comments}[1]{}

\begin{document}

\author[Wes Camp]{Wes Camp}
\address{Department of Mathematics\\
        Vanderbilt University\\
        Nashville, TN 37240}
\email{w.camp@vanderbilt.edu}

\title[]{Right-angled Artin groups with non-path-connected boundary}%

\begin{abstract}
We place conditions on the presentation graph $\Gamma$ of a right-angled Artin group $A_\Gamma$ that guarantee the standard CAT(0) cube complex on which $A_\Gamma$ acts geometrically has non-path-connected boundary.
\end{abstract}

\maketitle
%


\section{Introduction}

In \cite{Gr}, Gromov showed that if $G$ is a hyperbolic group acting geometrically on two metric spaces $X$ and $Y$, then the boundaries of $X$ and $Y$ are homeomorphic. The same is not true for CAT(0) spaces; in \cite{CK} Croke and Kleiner demonstrate a group that acts geometrically on two CAT(0) spaces with non-homeomorphic boundaries, and it was later shown (\cite{Wi}) that the same group has uncountably many distinct CAT(0) boundaries. The group is the right-angled Artin group whose presentation graph is the path on four vertices $P_4$, and so has presentation $$\langle a,b,c,d \mid [a,b]=[b,c]=[c,d]=1 \rangle.$$ In \cite{CMT}, it is shown that the boundary of the standard CAT(0) cube complex on which this group acts is non-path-connected. The boundary of such a cube complex is connected if and only if the the presentation graph of the group is connected (and so the group is one-ended). In this paper, the method in \cite{CMT} is generalized to a class of right-angled Artin groups whose presentation graphs admit a certain type of splitting. The main theorem here is as follows:

\begin{thm}
Let $\Gamma$ be a connected graph. Suppose $\Gamma$ contains an induced subgraph $(\{a,b,c,d\}, \{\{a,b\}, \{b,c\}, \{c,d\}\})$ (isomorphic to $P_4$), and there are subsets $B\subset lk(c)$ and $C \subset lk(b)$ with the following properties:
\begin{enumerate}
\item $B$ separates $c$ from $a$ in $\Gamma$, with $d\notin B$;
\item $C$ separates $b$ from $d$ in $\Gamma$, with $a\notin C$;
\item $B\cap C=\emptyset$.
\end{enumerate}

Then $\partial \mathcal{S}_\Gamma$ is not path connected.
\end{thm}

Here, $\SG$ is the standard CAT(0) cube complex on which the right-angled Artin group $A_\Gamma$ with presentation graph $\Gamma$ acts geometrically, and $lk(v)$ is the set of vertices of $\Gamma$ sharing an edge with $v$. We in fact show a slightly stronger result, with the hypothesis $B\cap C = \emptyset$ replaced with the statement of Claim \ref{empty}. The hypotheses here essentially require a copy of $P_4$ in $\Gamma$ that is either not contained in a cycle, or has every cycle containing it separated by chords based at $b$ and $c$. It is a known fact of graph theory that any graph that does not split as a join contains an induced subgraph isomorphic to $P_4$, and any graph $\Gamma$ that splits as a non-trivial join has $\partial \SG$ path connected, so the hypothesis that $\Gamma$ contain a copy of $P_4$ is satisfied in any interesting case. 

If a connected boundary of a CAT(0) space is locally connected, then it is a Peano space (a continuous image of [0,1]) and therefore path connected. The boundaries of some right-angled Coxeter groups are therefore known to be path connected (\cite{MRT} and \cite{CM}), because they are locally connected. However, a consequence of a theorem in \cite{MR} is that for right-angled Artin groups, $\partial \SG$ is locally connected iff $\Gamma$ is a complete graph; i.e. $A_\Gamma \cong \Integ^n$ and $\partial \SG \cong S^{n-1}$. Thus no approach involving local connectivity works for right-angled Artin groups.

In \cite{Mo}, the construction of \cite{CK} is generalized to demonstrate a class of groups with non-unique boundary. These groups are of the form $$G=(G_1 \times \Integ^n) *_{\Integ^n} (\Integ^n \times \Integ^m) *_{\Integ^m} (\Integ^m \times G_2),$$ where $G_1$ and $G_2$ are infinite CAT(0) groups. It is easily verified that if $G_1$ and $G_2$ are right-angled Artin groups, then $G$ is a right-angled Artin group whose presentation graph satisfies the conditions of the main theorem of this paper; in fact, the method of this paper should work even if $G_1$ and $G_2$ are arbitrary infinite CAT(0) groups.

It seems this boundary path connectivity problem may be related to the question of when two right-angled Artin groups are quasi-isometric. In \cite{BN}, Behrstock and Neumann show that all right-angled Artin groups whose presentation graphs are trees of diameter greater than 2 are quasi-isometric; in \cite{BKS}, Bestvina, Kleiner, and Sageev show that right-angled Artin groups with atomic presentation graphs (no valence 1 vertices, no separating vertex stars, and no cycles of length $\leq 4$) have $A_\Gamma$ quasi-isometric to $A_{\Gamma'}$ iff $\Gamma \cong \Gamma'$. The connection between these results and the result of this paper is that if $\Gamma$ is a tree of diameter greater than 2, then $\Gamma$ satisfies the hypotheses of the main theorem here, and therefore $\partial \SG$ has non-path-connected boundary; if $\Gamma$ is atomic, then $\Gamma$ cannot satisfy the hypotheses of the main theorem here. 

The author would like to thank Mike Mihalik for his guidance during the writing of this paper.
\section{Preliminaries}

\begin{defn}
Given a (undirected) graph $\Gamma$ with vertex set $S=a_1, \dots, a_n$, the corresponding \textbf{right-angled Artin group} $A_\Gamma$ is the group with presentation $$\langle a_1, \dots, a_n \mid [a_i, a_j] \text{ if } i<j \text{ and } \{a_i,a_j\} \text{ is an edge of } \Gamma \rangle.$$ We call $\Gamma$ the \textbf{presentation graph} for $A_\Gamma$.
\end{defn}

\begin{defn}
\label{rearr}
If $A_\Gamma$ is a right-angled Artin group with Cayley graph $\Lambda_\Gamma$, let $\overline e\in S$ be the label of the edge $e$ of $\Lambda_\Gamma$. An {\bf edge path}  $\alpha\equiv (e_1,e_2,\ldots, e_n)$ in $\Lambda_\Gamma$ is a map $\alpha:[0,n]\to \Lambda_\Gamma$ such that $\alpha$ maps $[i,i+1]$ isometrically to the edge $e_i$. For $\alpha$ an edge path in $\Lambda_\Gamma$, let $lett(\alpha)\equiv \{\overline e_1, \ldots , \overline e_n\}$, and let $\overline \alpha\equiv\overline e_1\cdots \overline e_n$. If $\beta$ is another geodesic with the same initial and terminal points as $\alpha$, then call $\beta$ a {\bf rearrangement} of $\alpha$.
\end{defn}

\begin{lemma}
\label{radel}
If $w=g_1 \dots g_k$ is a word in $A_\Gamma$ (with each $g_i \in S^\pm$) that is not of minimal length, then two letters of $g_1\dots g_k$ \textbf{delete}; that is, for some $i < j$, $g_i=g_j^{-1}$, the sets $\{g_i, g_j\}$ and $\{g_{i+1}, \dots, g_{j-1}\}$ commute, and $w=g_1 \dots g_{i-1} g_{i+1} \dots g_{j-1} g_{j+1} \dots g_k$.
\end{lemma}

\begin{proof} 
Let $w=h_1 \dots h_m$ be a minimal length word representing $w$, and draw a van Kampen diagram $D$ for the loop $g_1 \dots g_k h_m^{-1} \dots h_1^{-1}$. For each boundary edge $e_i$ corresponding to a $g_i$, trace a band across the diagram by picking the opposite edge of $e_i$ in the relation square containing $e_i$, and continuing to pick opposite edges (without going backwards). Note that such a band cannot cross itself, and so this band must end on another boundary edge of $D$. Since $k > m$, there is some boundary edge $e_i$ corresponding to some $g_i$ that has its band $B$ end on a boundary edge $e_j$ corresponding to $g_j$, with $i < j$. Note this implies $g_i = g_j^{-1}$. Now, either all the bands corresponding to $g_{i+1}, \dots, g_{j-1}$ cross $B$ (implying each of $g_{i+1}, \dots, g_{j-1}$ commutes with $g_i$ and $g_j$), or some band corresponding to one of $g_{i+1}, \dots, g_{j-1}$ ends on a boundary edge corresponding to another of $g_{i+1}, \dots, g_{j-1}$. Picking an ``innermost'' such band and repeating the above argument gives the desired result.
\end{proof}

\begin{remark}
\label{vkchange}
Note that the bands in the van Kampen diagram $D$ share the same labels along their `sides'. This means that deleting the band $B$ from the diagram and matching up the separate parts of what remains (along paths with the same labels) gives a van Kampen diagram $D'$ for the loop \\$w=g_1 \dots g_{i-1} g_{i+1} \dots g_{j-1} g_{j+1} \dots g_k h_m^{-1} \dots h_1^{-1}$.
\end{remark}

\begin{remark}
Given a non-geodesic edge path $(e_1, \dots, e_k)$ in the Cayley graph $\Lambda_\Gamma$ for $A_\Gamma$, we say edges $e_i$ and $e_j$ delete if their corresponding labels delete in the word $\overline{e_1}\dots\overline{e_k}$.
\end{remark}

\begin{lemma}
\label{diamond}
Suppose $A_\Gamma$ is a right-angled Artin group, and $(\alpha_1,\alpha_2)$ and $(\beta_1,\beta_2)$ are geodesics between the same two points in in the Cayley graph $\Lambda_\Gamma$ for $A_\Gamma$. There exist geodesics $(\gamma_1,\tau_1), (\gamma_1,\delta_1), (\delta_2,\gamma_2)$, and $(\tau_2,\gamma_2)$ with the same end points as $\alpha_1,\beta_1,\alpha_2,\beta_2$ respectively, such that:
\begin{enumerate}
\item $\tau_1$ and $\tau_2$ have the same labels,
\item $\delta_1$ and $\delta_2$ have the same labels, and
\item $lett(\tau_1)$ and $lett(\delta_1)$ are disjoint and commute.
\end{enumerate}
Furthermore, the paths $(\tau_1^{-1},\delta_1)$ and $(\delta_2,\tau_2^{-1})$ are geodesic. 
\end{lemma}

\begin{center}
\includegraphics{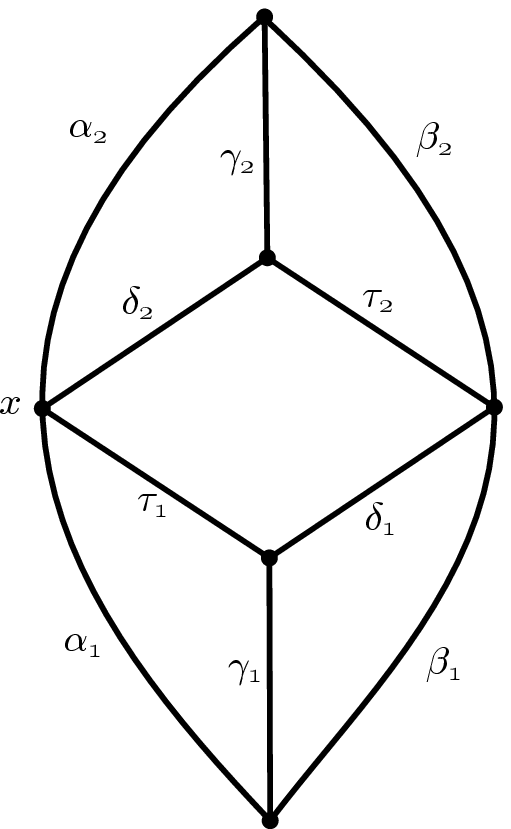}
\end{center}

\centerline{\textsc{Figure 1}}

\medskip

\begin{proof}
Let $D$ be a van Kampen diagram for the loop $(\alpha_1, \alpha_2, \beta_2^{-1}, \beta_1^{-1})$, and let $\alpha_1=(a_1, \dots, a_k)$, $\beta_1=(b_1, \dots, b_m)$. Let $a_{i_1}, \dots, a_{i_j}$ be (in order) the edges of $\alpha_1$ whose bands in $D$ end on $\beta_1$. Note that by Lemma \ref{radel}, $\beta_1$ can be rearranged to begin with an edge labeled $\overline{a_{i_1}}$, since $a_{i_1}$ and $b_{\ell_1}$ delete in $(\alpha_1^{-1}, \beta_1)$ for some $\ell_1$ and all the bands based at $b_1, \dots, b_\ell, a_1, \dots, a_{{i_1}-1}$ cross the band based at $a_{i_1}$ and ending at $b_{\ell_1}$. Similarly, $\beta_1$ can be rearranged to begin with an edge labeled $\overline{a_{i_1}}$ followed by an edge labeled $\overline{a_{i_2}}$, and continuing in this manner, we obtain a rearrangement of $\beta_1$ that begins with $\gamma_1=(\overline{a_{i_1}}, \dots, \overline{a_{i_j}})$, and we let $\delta_1$ be the remainder of this rearrangment. This argument also implies $\alpha_1$ can be rearranged to begin with $\gamma_1$, and we let $\tau_1$ be the remainder of this rearrangement. Note that if $e$ is an edge of $\tau_1$, no edge of $\delta_1$ is labeled $\overline{e}$ or $\overline{e}^{-1}$, since bands with those labels must have crossed in $D$. We obtain $\gamma_2, \tau_2$ and $\delta_2$ in the analogous way from $\alpha_2$ and $\beta_2$, and note that in a van Kampen diagram $B'$ for $(\tau_1, \delta_2, \tau_2^{-1}, \delta_1^{-1})$, no band based on $\tau_1$ can end on $\delta_2$, since $(\tau_1, \delta_2)$ is geodesic, and no band based on $\tau_1$ ends on $\delta_1$, since $\tau_1$ and $\delta_1$ share no labels or inverse labels. Therefore all bands on $\tau_1$ end on $\tau_2$, so $\tau_1$ and $\tau_2$ have the same labels, as do $\delta_1$ and $\delta_2$.
\end{proof}

\begin{defn}
Under the hypotheses of the previous lemma, we call $\tau_1$ the \textbf{down edge path} at $x$, and we call $\delta_2$ the \textbf{up edge path} at $x$. If $\alpha_1$ and $\beta_1$ have the same length, we call the above figure the \textbf{diamond} at $x$ for $(\alpha_1, \alpha_2)$ and $(\beta_1, \beta_2)$.
\end{defn}

\begin{defn}
$P_4$ is the (undirected) graph on four vertices $a,b,c,d$, with edge set $\{\{a,b\},\{b,c\},\{c,d\}\}$.
\end{defn}

\begin{defn}
The \textbf{union} of two graphs $(V_1, E_1)$ and $(V_2, E_2)$ is the graph $(V_1 \cup V_2, E_1 \cup E_2)$.
\end{defn}

\begin{defn}
The \textbf{join} of two graphs $(V_1, E_1)$ and $(V_2, E_2)$ is the graph $(V_1 \cup V_2, E_1 \cup E_2 \cup (V_1\times V_2))$.
\end{defn}

\begin{defn}
A graph is \textbf{decomposable} if it can be expressed as joins and unions of isolated vertices.
\end{defn}

The following is Theorem 9.2 in \cite{Mer}.

\begin{thm}
A finite graph $G$ is decomposable iff it does not contain $P_4$ as an induced subgraph.
\end{thm}

In particular, if a connected graph $G$ does not contain $P_4$ as an induced subgraph, then it must split as the join $G_1 \vee G_2$, for some subgraphs $G_1, G_2$ of $G$. 

\begin{defn}
For a graph $\Gamma$ and a vertex $a$ of $\Gamma$, $lk(a)=\{b \in \Gamma \mid \{a,b\} \text{ is an edge of } \Gamma\}$.
\end{defn}

Let $\Lambda_\Gamma$ be the Cayley graph for the group $A_\Gamma$.

\begin{defn}
The \textbf{standard complex} $\mathcal{S}_\Gamma$ for the group $A_\Gamma$ is the CAT(0) cube complex whose one-skeleton is $\Lambda_\Gamma$, with each cube given the geometry of $[0,1]^n$ for the appropriate $n$.
\end{defn}

For more on cube complexes and the definitions below, see \cite{Sag}.

\begin{defn}
A \textbf{midcube} in a cube complex $C$ is the codimension 1 subspace of an $n$-cube $[0,1]^n$ obtained by restricting exactly one coordinate to $\frac{1}{2}$. A \textbf{hyperplane} is a connected nonempty subspace of $C$ whose intersection with each cube is either empty or consists of one of its midcubes.
\end{defn}

\begin{lemma}
\label{hypsplit}
If $D$ is a hyperplane of the cube complex $C$, then $C-D$ has exactly two components.
\end{lemma}

Given a graph $\Gamma$, a vertex $v$ of $\Gamma$, and the corresponding standard complex $\SG$, note that if a hyperplane of $\SG$ intersects an edge of $\SG$ with label $v$, then every edge intersected by this hyperplane is also labeled $v$. Thus we can refer to hyperplanes in $\SG$ as $v$-hyperplanes, for $v$ a vertex of $\Gamma$. If $x$ is a vertex of $\SG$, then $xv$ and $x$ are separated by a $v$-hyperplane $D$. Let $x\mathcal{S}_{lk(v)}$ denote the cube complex generated by the coset $x\langle lk(v) \rangle$; then $D$ and $x\mathcal{S}_{lk(v)}$ are isometric and parallel, of distance $\frac{1}{2}$ apart.

\begin{defn}
A metric space $(X,d)$ is \textbf{proper} if each closed ball is compact.
\end{defn}

\begin{defn}
Let $(X,d)$ be a proper CAT(0) space. Two geodesic rays $c,c':[0,\infty)\rightarrow X$ are called \textbf{asymptotic} if for some constant $K$, $d(c(t),c'(t))\leq K$ for all $t\in[0,\infty)$. Clearly this is an equivalence relation on all geodesic rays in $X$. We define the \textbf{boundary} of $X$ (denoted $\partial X$) to be the set of equivalence classes of geodesic rays in $X$. We denote the union $X\cup\partial X$ by $\overline X$.
\end{defn}

The next proposition guarantees that the topology we wish to put on the boundary is independent of our choice of basepoint in $X$.

\begin{prop}
Let $(X,d)$ be a proper CAT(0) space, and let $c:[0,\infty)\rightarrow X$ be a geodesic ray. For a given point $x\in X$, there is a unique geodesic ray based at $x$ which is asymptotic to $c$.
\end{prop}

For a proof of this (and more details on what follows), see \cite{BH}.

We wish to define a topology on $\overline X$ that induces the metric topology on $X$. Given a point in $\partial X$, we define a neighborhood basis for the point as follows: 

\noindent Pick a basepoint $x_0\in X$. Let $c$ be a geodesic ray starting at $x_0$, and let $\epsilon>0$, $r>0$. Let $S(x_0,r)$ denote the sphere of radius $r$ centered at $x_0$, let $B(x_0, r)$ denote the open ball of radius $r$ centered at $x_0$ and let $p_r:X-B(x_0,r)\rightarrow S(x_0,r)$ denote the projection to $S(x_0,r)$. Define
\begin{center}
$U(c,r,\epsilon)=\{x\in\overline X :d(x,x_{0})>r,d(p_{r}(x),c(r))<\epsilon\}.$
\end{center}
This consists of all points in $\overline X $ whose projection to $S(x_0,r)$ is within $\epsilon$ of the point of the sphere through which $c$ passes. These sets together with the metric balls in $X$ form a basis for the \textbf{cone topology}. The set $\partial X$ with this topology is sometimes called the \textbf{visual boundary}. In this article, we will call it the boundary of $X$.

\begin{prop}
If $X$ and $Y$ are proper CAT(0) spaces, then $\partial(X \times Y) \cong \partial X * \partial Y$, where $*$ denotes the spherical join.
\end{prop}

If the graph $\Gamma$ splits as a non-trivial join $\Gamma_1 \vee \Gamma_2$, then the group $A_\Gamma$ splits as the direct product $A_{\Gamma_1} \times A_{\Gamma_2}$, and so we have $\mathcal{S}_\Gamma \cong \mathcal{S}_{\Gamma_1} \times \mathcal{S}_{\Gamma_2}$. The previous proposition then gives that $\partial \mathcal{S}_\Gamma \cong \partial \mathcal{S}_{\Gamma_1} * \partial \mathcal{S}_{\Gamma_2}$. Any non-trivial spherical join is path connected, and so $\partial \mathcal{S}_\Gamma$ is path connected. 

\begin{lemma}
\label{track}
There is a bound $\delta >0$ such that if $\alpha$ is a CAT(0) geodesic path in $\SG$, then there is a Cayley graph geodesic path $\beta$ in $\Lambda_\Gamma$ (contained naturally in $\SG$) such that each vertex of $\beta$ is within distance $\delta$ of $\alpha$, and each point of $\alpha$ is within $\delta$ of a vertex of $\beta$.
\end{lemma}

A proof of this can be found in Section 3 of \cite{HW}.

\section{Result}

The goal of this section is to prove the following theorem:

\begin{thm}
\label{main}
Let $\Gamma$ be a connected graph. Suppose $\Gamma$ contains an induced subgraph $(\{a,b,c,d\}, \{\{a,b\}, \{b,c\}, \{c,d\}\})$ (isomorphic to $P_4$), and there are subsets $B\subset lk(c)$ and $C \subset lk(b)$ with the following properties:
\begin{enumerate}
\item $B$ separates $c$ from $a$ in $\Gamma$, with $d\notin B$;
\item $C$ separates $b$ from $d$ in $\Gamma$, with $a\notin C$;
\item $B\cap C=\emptyset$.
\end{enumerate}

Then $\partial \mathcal{S}_\Gamma$ is not path connected.
\end{thm}

In fact, we prove a stronger result, with the hypothesis $B\cap C = \emptyset$ replaced by the statement of Claim \ref{empty}. For the remainder of this section, suppose $a,b,c,d \in \Gamma$, $B\subset lk(c)$, and $C \subset lk(b)$ are as in Theorem \ref{main}. Note that $b\in B$, $c\in C$. We wish to consider the following rays in $\Lambda_\Gamma$ (equivalently in $\SG$), based at the identity vertex $*$:

$$r=cdab(cb)^2cdab(cb)^6\dots=\prod_{i=1}^{\infty}(cb)^{k_i}cdab$$

and

$$s=dbcb^2adbc(b^2c)^2b^2adbc(b^2c)^6b^2a\dots=\prod_{i=1}^{\infty}dbc(b^2 c)^{k_i}b^2 a$$

where the $k_i$ are defined recursively with $k_0=-1$, $k_{i+1}=2k_i + 2$.

Define the following vertices of $r$, for $n\geq 0$:

$$v_n = \left(\prod_{i=1}^{n}(cb)^{k_i}cdab\right)(cb)^{k_{n+1}}cd$$ 

$$v_n'=v_na$$

Define the following vertices of $s$, for $n\geq 0$:

$$w_n = \left(\prod_{i=1}^{n}dbc(b^2 c)^{k_i}b^2 a\right)$$

$$w_n'=w_nd$$

\vspace{2mm}

\begin{center}
\includegraphics[scale=.83]{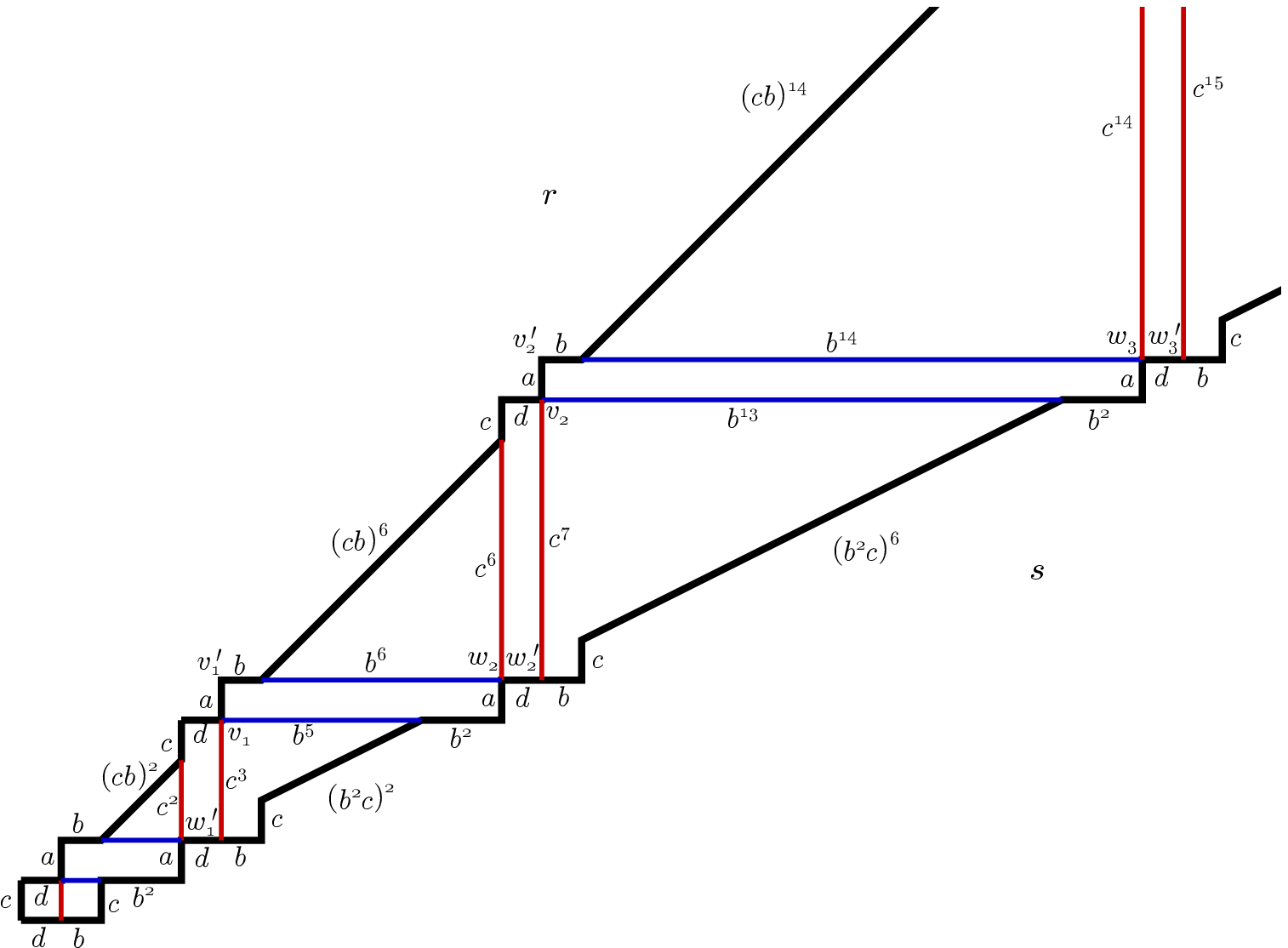}
\end{center}

\centerline{\textsc{Figure 2}}

\medskip
We have $v_0 = cd$, $v_0' = cda$, $v_1 = cdab(cb)^2cd$, $w_0 = *$, $w_0'=d$, $w_1 = dbcb^2a$. It will be helpful to refer to Figure 2 for many of the claims that follow.

The following is proved in \cite{CMT}.

\begin{claim}
\label{lines}
For $n \geq 0$, $v_n=w_n'c^{k_{n+1}+1}$ and $v'_{n}b^{k_{n+2}+1}=w_{n+1}$.

\end{claim}

Since $b\in B$ and $c\in C$, we then have $v_n \C = w_n' \C$ and $w_n \B = v'_{n-1} \B$.

If $Q_c$ denotes the component of $c$ in $\Gamma - B$, and $Q_b$ denotes the component of $b$ in $\Gamma - C$, then $A_\Gamma$ can be represented as $\langle Q_c \cup B \rangle *_{B} \langle \Gamma - Q_c\rangle$ or $\langle Q_b \cup C \rangle *_{C} \langle \Gamma - Q_b\rangle$, and so at each vertex $x$ of $\Lambda_\Gamma$, the cosets $x\B$ and $x\C$ separate $\Lambda_\Gamma$. Therefore, if $x\SB$ and $x\SC$ denote the cube complexes generated by $\B$ and $\C$ respectively at a vertex $x$ of $\SG$, then $x\SB$ and $x\SC$ separate $\SG$. Note that $\SG - x\SB$ has at least two components: one containing $xc^{-1}$, and one containing $xa$. Similarly, $\SG - x\SC$ has at least two components: one containing $xb^{-1}$, and one containing $xd$.

For each $i$, define the following components of $\SG$:
\begin{enumerate}
\item $V_i^+$ is the component of $\SG - v_i\SB$ containing $v_ia$;
\item $V_i^-$ is the component of $\SG - v_i\SB$ containing $v_ic^{-1}$;
\item $W_i^+$ is the component of $\SG - w_{i}\SC$ containing $w_{i}d$;
\item $W_i^-$ is the component of $\SG - w_{i}\SC$ containing $w_{i}b^{-1}$.
\end{enumerate}
Note $V_i^+$ contains the vertices of $r$ after $v_i$, and $W_i^+$ contains the vertices of $s$ after $w_{i}$. For each $V_i^\pm$, (respectively $W_i^\pm$), let $\overline{V_i^\pm}$ denote the closure of $V_i^\pm$ in $\SG$, so $\overline{V_i^\pm}=V_i^\pm \cup v_i\SB$ $(\overline{W_i^\pm}=W_i^\pm \cup w_{i}\SC)$. For a subset $S$ of $\SG$, let $L(S)$ denote the limit set of $S$ in $\partial \SG$.

\begin{claim}
\label{sep1}
\begin{enumerate}
\item The sets $\overline{V_i^\pm}$, $\overline{W_i^\pm}$ are convex.
\item $L(\overline{V_i^+})\cap L(\overline{V_i^-}) = L(v_i\SB)$ and $L(\overline{W_i^+})\cap L(\overline{W_i^-}) = L(w_{i}\SC)$.
\item The set $L(v_i\SB)$ (respectively $L(w_{i}\SC)$) separates $L(\overline{V_i^+})$ and $L(\overline{V_i^-})$ (respectively $L(\overline{W_i^+})$ and $L(\overline{W_i^-})$) in $\partial X$.
\end{enumerate}
\end{claim}

\begin{proof}
For (1), the only way out of the set $\overline{V_i^+}$ is through the convex subcomplex $v_i\SB$.

For (2), if $q$ is a ray in $L(\overline{V_i^+})\cap L(\overline{V_i^-})$, then there are geodesic rays $q_1 \in \overline{V_i^+}$, $q_2 \in \overline{V_i^-}$ that are a bounded distance from $q$, and therefore from one another. Thus both $q_1$ and $q_2$ remain a bounded distance from $v_i\SB$, as required.

For (3), suppose $\alpha: [0,1]\rightarrow \partial \SG$ is a path connecting $x \in L(V_i^+)$ and $y \in L(V_i^-)$. Choose $w \in v_i\SB$, and for each $t\in[0,1]$, let $\beta_t:[0, \infty)\rightarrow \SG$ be the geodesic ray from $w$ to $\alpha(t) \in \partial \SG$. This gives a continuous map $H:[0,1]\times [0, \infty) \rightarrow \SG$ where $H(t,s) = \beta_t(s)$. Note $H(0, s) \subset V_i^+$, $H(1, s)\subset V_i^-$. For each $n \geq 0$, let $z_n$ be a point of $H([0,1]\times \{n\})$ in $v_i\SB$; then $L(\cup_{n=1}^{\infty} \{z_n\}) \subset Im(\alpha) \cap L(v_i\SB)$ as required.
\end{proof}

In \cite{CMT}, it is shown that $r$ and $s$ track distinct CAT(0) geodesics in $\SG$, so $L(r)$ and $L(s)$ are distinct one-element sets.

\begin{claim}
For $n\geq 1$, the sets $L(w_{2n-1}\SC)$ and $L(r)$ are separated in $\partial \SG$ by $L(v_{2n+1}\SB)$.
\end{claim}

\begin{proof}
First note that $L(r)\in L(V_{i}^+)$ for each $i\geq 1$. Let $D_{2n}$ be the $d$-hyperplane that separates $w_{2n}$ from $w_{2n}'$ (and also separates $v_{2n}$ from the previous vertex of $r$), and let $A_{2n}$ be the $a$-hyperplane that separates $v_{2n}$ from $v_{2n}'$ (and also separates $w_{2n+1}$ from the previous vertex of $s$). Note that $w_{2n-1}\SC$ is contained in the same component of $\SG-D_{2n}$ as $*$ since $d \notin C$ and therefore no path in $\C$ based at $w_{2n-1}$ crosses $D_{2n}$. Also note $A_{2n} \subset V_{2n+1}^-$. Since $D_{2n}$ and $A_{2n}$ cannot cross (since $d$ does not commute with $a$), and $D_{2n}$ is not in the same component as $v_{2n+1}\SB$ in $\SG-A_{2n}$, we have that $w_{2n-1}\SC \subset V_{2n+1}^-$. The previous claim gives the result.
\end{proof}

\begin{claim}
For $n\geq 1$, the sets $L(v_{2n-1}\SB)$ and $L(r)$ are separated in $\partial \SG$ by $L(w_{2n+1}\SC)$.
\end{claim}

\begin{proof}
The proof is analagous to the proof of the previous claim, replacing the hyperplanes $D_{2n}$ and $A_{2n}$ with the hyperplanes $A_{2n-1}$ and $D_{2n}$ respectively.
\end{proof}

\begin{remark}
\label{bounce}
The previous two claims imply that if there is a path in $\partial \SG$ between a point of $L(w_1\SC)$ and $L(r)$, the path must pass through (in order) $L(v_{3}\SB)$, $L(w_{5}\SC)$, $L(v_{7}\SB)$, $L(w_{9}\SC)$, and so on. 
\end{remark}

We will now show that the sets $L(v_i\SB)$ (resp. $L(w_i\SC)$) are eventually `close' to $L(s)$ (resp. $L(r)$), implying the path described in Remark \ref{bounce} cannot exist. 

\begin{claim}
\label{empty}
$C\cap lk(a) \cap lk(d) = C \cap lk(a) \cap lk(c) = \emptyset$, and $B\cap lk(a) \cap lk(d) = B \cap lk(d) \cap lk(b) = \emptyset$.
\end{claim}

\begin{proof}
If $e \in C\cap lk(a) \cap lk(d)$, then $(a, e, d, c)$ is a path from $a$ to $c$ in $\Gamma$. Since $B$ separates $a$ from $c$ and $d \notin B$, we must have $e \in B$, but $B\cap C = \emptyset$. Similarly, if $e \in C \cap lk(a) \cap lk(c)$, then $(a, e, c)$ is a path from $a$ to $c$ in $\Gamma$, and so $e \in B$, contradiction. The remaining statements are proved identically.
\end{proof}

For $i\geq 1$, let $r_i$ (respectively $s_i$) be the segment of $r$ (respectively $s$) between $*$ and $v_i'$ (respectively $*$ and $w_i'$). Let $\beta_i$ be a Cayley graph geodesic ray based at $w_i'$ with labels in $B$, and let $\gamma_i$ be a Cayley graph geodesic ray based at $v_i'$ with labels in $C$. 

\begin{claim}
\label{close}
Any Cayley graph geodesic from $*$ to a point of $\gamma_i$ must pass within 4 units of $v_{i}'$. Any Cayley graph geodesic from $*$ to a point of $\beta_i$ must pass within 4 units of $w_i'$.
\end{claim}

\begin{proof}
First observe that if $(r_i, \gamma_i)$ is not $\Lambda_\Gamma$-geodesic, then an edge of $\gamma_i$ must delete with an edge of $r_i$. Since $a,b,d \notin C$, the labels of these deleting edges must be $c$ and $c^{-1}$. However, the labels of these edges must also be in $lk(a) \cap lk (d)$, by Lemma \ref{radel} (see Figure 2). Therefore $(r_i, \gamma_i)$ is a Cayley geodesic.

Now, suppose there is a $\Lambda_\Gamma$-geodesic $\rho$ between $*$ and a point of $\gamma_i$ with $d(\rho, v_i') > 4$. Let $\alpha$ denote the segment of $(r_i, \gamma_i)$ between $*$ and the endpoint of $\rho$. Consider a diamond based at $v_i'$ for $\rho$ and $\alpha$ as in Lemma \ref{diamond}. Let $\tau$ and $\delta$ be the down edge path and up edge path respectively at $v_i'$, and note $\tau$ and $\delta$ have length at least 3. Every $\Lambda_\Gamma$-geodesic from $*$ to $v_i'$ must end with an edge labeled $a$, so every label of $\delta$ is in $lk(a)$. If an edge of $\tau$ has label $d$, then every label of $\delta$ is in $C \cap lk(a) \cap lk(d)$, but this set is empty by Claim \ref{empty}. By Lemma \ref{radel} every other edge of $\tau$ has its label in $lk(d)\cap\{a,b,c,d\}$, so the remaining edges of $\tau$ must be labeled $c$, but $C \cap lk(a) \cap lk(c)$ is also empty. Thus $d(\rho, v_i') \leq 4$. The proof of the second statement is identical.
\end{proof}

\begin{claim}
$\partial \SG$ is not path connected.
\end{claim}
\begin{proof} 
Observe that since $v'_{n-1}b^{k_{n+1}+1}=w_n$ by Claim \ref{lines} and $C \subset lk(b)$, any ray $\alpha$ based at $w_n$ with labels in $C$ stays a bounded distance from the ray based at $v_{n-1}'$ with the same labels. Combining Claim \ref{close} and Lemma \ref{track}, we have that a CAT(0) geodesic from $*$ to a point of $L(\alpha)$ must pass within $\delta + 4$ of $v_{n-1}'$, where $\delta$ is the tracking constant given by Lemma \ref{track}.  We therefore have that any sequence of points $\{p_i\}_{i=1}^\infty$ with each $p_i \in L(w_i\SC) \subset \partial \SG$ must converge to $L(r)\in \partial \SG$. Similarly, any sequence of points $\{q_i\}_{i=1}^\infty$ with each $q_i \in L(v_i\SB) \subset \partial \SG$ must converge to $L(s) \in \partial \SG$. Therefore, by Remark \ref{bounce}, given any $\epsilon$, any path from a point of $L(w_1\SC)$ to $L(r)$ eventually bounces back and forth infinitely between the $\epsilon$-neighborhood of $L(s)$ and the $\epsilon$-neighborhood of $L(r)$, which is impossible; therefore, no such path exists.
\end{proof}

\newpage


\end{document}